\RequirePackage{fix-cm}
\documentclass[smallextended]{svjour3}  
\smartqed  
\usepackage{graphicx}
\usepackage[utf8]{inputenc}
\usepackage[english]{babel}
\usepackage{bbm}

\usepackage[nottoc]{tocbibind}
\usepackage{mathrsfs,amsfonts,amssymb,amsmath}

\usepackage{amsthm}
\usepackage{enumerate}
\usepackage{graphicx,cite}
\usepackage{romannum}
\usepackage{hyperref}
\usepackage{authblk}
\usepackage{graphicx}

\allowdisplaybreaks

\hypersetup{colorlinks=true, linkcolor=blue, citecolor=red}

\numberwithin{equation}{section}

\numberwithin{theorem}{section}
\numberwithin{remark}{section}

\usepackage[square,numbers]{natbib}
\begin{document}
\pagenumbering{arabic}
\setlength{\belowdisplayskip}{0pt}

\newpage

\title{Polynomial escape rates via maximal large deviations}


\author{
        Yaofeng Su
}


\institute{
 Yaofeng Su  \at
              Università degli Studi di Roma Tor Vergata, Italy\\
              \email{suyaofengmath@gmail.com} 
}


\maketitle

\begin{abstract}
In this short note, we propose a new and short approach to polynomial escape rates, which can be applied to various open systems with intermittency. The tool of our approach is the maximal large deviations developed in \cite{mldp}.

\keywords{escape rates \and open dynamical systems \and intermittent maps \and Maximal large deviations}

\end{abstract}

\section{introduction}

In the early 80s, mathematicians suggested investigating open systems (i.e., systems with holes or leakages) as a means
of generating transient chaos \cite{york},  retrieving information from a distribution and inferring properties of the equivalent closed systems (i.e. systems without holes). More precisely, consider a dynamical system $(X, f, \mu)$, $\mu$ is an invariant probability on $X$ preserved by the dynamics $f$, a hole is a subset $H \subsetneq X$, then the phase space $X$ is opened by $H$, leading to leakages/escapes of trajectories through this hole $H$. Such a system is called an open system. An escape time $e_H$ is defined by $e_H(x):=\inf\{n\ge 0: f^n(x) \in H\}$. A first hitting time  $\tau_H$ is defined by $\tau_H(x):=\inf\{n\ge 1: f^n(x) \in H\}$, its distribution $\mu(\tau_H>n)$ represents the probability of trajectories that remain within $X$ up to time $n$. Studying the decay rate of $\mu(\tau_H>n)=\mu(e_H>n-1)$, referred to as an escape rate, can reveal insights into the chaotic behavior of the closed system $(X, f, \mu)$.

For closed systems with exponential mixing rates, such as some uniformly expanding maps and Sinai billiards, exponential escape rates have been demonstrated in \cite{Demersexp, haydn, fresta, escapeinfinitesinai,escapesinai}. 
However, for the closed systems enjoying polynomial mixing rates, there are few rigorous results in this direction:  Liverani-Saussol-Vaienti maps were examined in \cite{Demers}, where a polynomial escape rate was established. 

The purpose of the present paper is to propose a short but new approach to polynomial escape rates,  which can be applied to various systems with intermittency. Our results and methods have the following new aspects. 
\begin{enumerate}
    \item We prove that a Young tower with polynomial mixing rates has a polynomial escape rate (same as the mixing rate, see Remark \ref{adapt3}). This is different from the general results in \cite{haydn, fresta} which proved that a polynomial $\phi$-mixing dynamical system has an exponential escape rate. 
    \item When $\mu$ is an SRB measure,  our Theorem \ref{thm} recovers one result in \cite{Demers}. But the requirements of our hole $H$ in Theorems \ref{thm} and \ref{thm1} are relaxed and are not required to be a Markov hole as in \cite{Demers}.
    \item While our Theorems \ref{thm} and \ref{thm1} focus on specific systems, our methods can be adapted to various dynamical systems with intermittency, see Remarks \ref{adapt3}, \ref{adapt1} and \ref{adapt2}.
    \item In contrast to the exponential escape rates for two-dimensional Sinai billiards \cite{escapeinfinitesinai}, our Theorem \ref{thm1} presents polynomial escape rates for a two-dimensional hyperbolic system with intermittency.
\end{enumerate}

The structure of this paper: section \ref{def} includes necessary definitions and notations used throughout this paper. Section \ref{1d} addresses polynomial escape rates for one-dimensional intermittent maps. Its method is adapted to polynomial escape rates for two-dimensional hyperbolic systems with intermittency in Section \ref{highd}.

\section{Definitions and notations}\label{def}

Throughout this paper, we consider a dynamical system $(X, f, \mu)$, $f$ is a hyperbolic dynamics, and the probability $\mu$ preserved by $f$ is an SRB measure on $X$, i.e., the conditional measure on unstable manifolds is absolutely continuous w.r.t. the Lebesgue measure on the unstable manifolds. A hole $H \subsetneq X$ satisfies $\mu(H)\in (0,1)$. 

The notation $cls(H)$ (resp. $int(H)$) means the closure (resp. interior) of $H$, $a_n \precsim b_n$ (resp. $a_n \succsim b_n$) means that there is a constant $C\ge1$ such that $|a_n|\le C |b_n|$ (resp. $|a_n|\ge C^{-1} |b_n|$). $a_n \approx b_n$ means that $|a_n|\precsim |b_n|$ and $|a_n|\succsim|b_n|$. 

\section{One-dimensional intermittent maps}\label{1d}

In this section, we present an approach to polynomial escape rates for one-dimensional intermittent maps. The Liverani-Saussol-Vaienti maps (LSV maps), which are typical examples of one-dimensional intermittent maps, are defined as follows:
\begin{eqnarray}\label{lsvmap}
g_{\alpha}(x) =
\begin{cases}
x+2^{\alpha}x^{1+\alpha},      & 0\le x \le \frac{1}{2}\\
2x-1,  & \frac{1}{2} < x \le 1 \\
\end{cases},\quad  0< \alpha <1 .
\end{eqnarray}

It was proved in \cite{lsv, Young} that there is an SRB measure $\mu_{\alpha}$ preserved by $g_{\alpha}$ and the mixing rate is $O(n^{1-1/\alpha})$.

\begin{theorem}[Escape Rates I]\ \label{thm}\ \par
 For the dynamical system $(X,f,\mu)$:= $([0,1],g_{\alpha}, \mu_{\alpha})$, if a (dis)connected hole $H$ satisfies $ 0\notin  cls(H)$ and $int(H)\neq \emptyset$, then  $\mu(\tau_H>n)\approx n^{1-1/\alpha}$.
 \end{theorem}

 \begin{proof}
     Choose a sufficiently small $a_m:=g_{\alpha}|_L^{-m}\{1\}$  such that $H \subseteq (a_m,1]$. Here $g_{\alpha}|_L$ is the left branch of $g_{\alpha}$.
     
     \vspace*{4pt}\noindent\textbf{1. A first return tower.} It follows from \cite{Young} that the system $(X, f, \mu)$ can be modelled by a first return Young tower $\Delta:=\{(x,n) \in (a_m,1] \times \{0,1,2, \cdots\}: n<R(x)\}\}$. Here $R(x):=\inf\{n\ge 1: f^n(x)\in (a_m,1]\}$ defined on $(a_m,1]$ is a first return time, the dynamics  $F:\Delta \to \Delta$ sends $(x,n)$ to $(x,n+1)$ if $n+1<R(x)$, and $(x,n)$ to $(f^R(x),0)$ if $n=R(x)-1$. $F$ preserves an invariant probability $\mu_{\Delta}$ on $\Delta$ and its mixing rate is $O(n^{1-1/\alpha})$. Let $\pi:\Delta \to X$, $\pi(x,n)=f^n(x)$, then $\pi_* \mu_{\Delta}=\mu$. \\
     \vspace*{4pt}\noindent\textbf{2. Lifting.} Define $\tau_{H\times \{0\}}:=\inf\{n\ge 1: F^n \in H\times \{0\}\}$, then 
     \begin{align*}
         \mu(\tau_H>n)&=\mu_{\Delta}(\inf\{n\ge 1: f^n \circ \pi \in H\}>n)=\mu_{\Delta}(\inf\{n\ge 1: F^n \in \pi^{-1}H\}>n)\\
         &=\mu_{\Delta}(\inf\{n\ge 1: F^n \in H\times \{0\}\}>n)=\mu_{\Delta}(\tau_{H\times \{0\}}>n).
     \end{align*}
     \vspace*{4pt}\noindent\textbf{3. Lower bounds.} Note that $\bigcup_{i \ge 1}\{(x,i): R(x)>n+i\} \subseteq \{\tau_{H\times \{0\}}>n\}$. Then $\mu(\tau_H>n)=\mu_{\Delta}(\tau_{H\times \{0\}}>n)\succsim \sum_{i}\mu(R> n+i) \approx n^{1-1/\alpha}$. We remark that this is also proved in Corollary 5.1 of \cite{ivan}.\\
     \vspace*{4pt}\noindent\textbf{4. Several hitting times.} Denote the base $(a_m,1]\times \{0\}$ by $B$, define a new hitting time on $\Delta$ by  $\tau_{H, R}:=\sum_{m= 0}^{\tau_{H\times \{0\}}}\mathbbm{1}_{B}\circ F^m$, and define another hitting time on $B$ by $\tau_{H, B}:=\tau_{H,R}|_{B}$. Since $f^R$ is a uniformly expanding map,  by Theorem 2.1 of \cite{Demersexp} $\mu_{\Delta}|_B(\tau_{H,B}>n)\precsim e^{-C n}$ for some $C=C_H>0$ when $H$ is a small Markov hole (i.e., an element of $\bigvee_{i=0}^{\infty} (f^R)^{-i}\{[a_m,a_{m-1}],\cdots, [a_2,a_1], [a_1,1]\}$). But this also holds for our hole $H$ because $int(H)$ must contain a small Markov hole. \\ 
     \vspace*{4pt}\noindent\textbf{5. Upper bounds.} By maximal large deviations, i.e., Theorem 3.1 of \cite{mldp}, we have \begin{gather*}
         \mu_{\Delta}\Big\{\max_{N\ge n}\Big|\frac{\sum_{m= 0}^{N}\mathbbm{1}_{B}\circ F^m}{N}-\mu_{\Delta}(B)\Big|\ge \frac{\mu_{\Delta}(B)}{2} \Big\}\precsim n^{1-1/\alpha}.
     \end{gather*} Denote $\mu_{\Delta}(B)$ by K, now we can estimate an upper bound: $\mu_{\Delta}(\tau_{H\times 0}>n)$\begin{align*}
         &\le \mu_{\Delta}\Big(\tau_{H\times \{0\}}>n, \Big|\frac{\tau_{H,R}}{\tau_{H\times \{0\}}}-K\Big|\ge \frac{K}{2}\Big) +\mu_{\Delta}\Big(\tau_{H\times \{0\}}>n, \Big|\frac{\tau_{H,R}}{\tau_{H\times \{0\}}}-K\Big|\le \frac{K}{2}\Big)\\
         & \le \mu_{\Delta}\Big\{\max_{N\ge n}\Big|\frac{\sum_{m= 0}^{N}\mathbbm{1}_{B}\circ F^m}{N}-K\Big|\ge \frac{K}{2} \Big\}+\mu_{\Delta}\Big(\tau_{H,R}>n \frac{K}{2}\Big)\\
         &\precsim n^{1-1/\alpha}+\sum_{i\ge 0}\mu_{\Delta}\Big\{(x,i):\tau_{H,R}(x,i)>n \frac{K}{2}\Big\}\\
         &\precsim n^{1-1/\alpha}+\sum_{i\ge 1}\mu_{\Delta}|_{B}\Big\{\tau_{H,B}>n \frac{K}{2}, R\ge i\Big\}\\
         &\precsim n^{1-1/\alpha}+\sum_{1\le  i\le nK/2}\mu_{\Delta}|_{B}\Big\{\tau_{H,B}>n \frac{K}{2}\Big\}+\sum_{i\ge nK/2}\mu_{\Delta}|_{B}\big\{ R\ge i\big\}\\
         &\precsim n^{1-1/\alpha}+n e^{-CKn/2}+n^{1-1/\alpha}\precsim n^{1-1/\alpha},
     \end{align*}where the fourth ``$\precsim$" is due to step 4. Hence we conclude the proof with these five steps.
\end{proof}

\begin{remark}\label{adapt3}
Actually we proved an abstract result: it follows from the proof of Theorem \ref{thm} that if the hole $H$ has an interior and is in the base of a one-dimensional Young tower with polynomial mixing rates, then the escape rate is the same as the mixing rate.

For the high-dimensional open dynamical systems that can be modeled by a first return polynomial Young tower $\Delta$ with a hole $H$ in the base $B$, we note from our proof of Theorem \ref{thm} that $\mu_{\Delta}(\tau_{H \times \{0\}})\succsim n^{1-1/\alpha} $ still holds, and $\mu_{\Delta}(\tau_{H \times \{0\}})\precsim n^{1-1/\alpha} $ holds provided that $\mu_{\Delta}|_B(\tau_{H,B}>n)\precsim e^{-C n}$ where $\tau_{H,B}:=\inf\{n\ge 1: f^R \in H\}$. This can be proved for some uniformly expanding open Markov maps by choosing a suitable Banach space to prove uniform Lasota-Yorke inequalities (e.g., generalized bounded variation functions) instead of the bounded variation functions for one-dimensional dynamics in \cite{Demersexp}.
\end{remark}

\begin{remark}\label{adapt1}
    Our proof can be easily extended to various intermittent systems. General statements and refined results can be found in \cite{WOPG}.  Here is one example, a finite-type $\alpha$-Farey map of exponent $\theta$ in \cite{farey}: \begin{eqnarray}\label{farey}
F_{\alpha}(x) =
\begin{cases}
(1-x)/a_1,      & x \in [t_2,t_1]\\
a_{n-1}(x-t_{n+1})/a_n+t_n,  & x \in [t_{n+1}, t_n], n \ge 2 \\
0,  & x =0 \\
\end{cases},
\end{eqnarray} where $a_n \ge 0, t_n=\sum_{i\ge n}a_n, t_1=1$ and $t_n\approx n^{-\theta}, \theta>1$. Following the same argument in the proof of Theorem \ref{thm}, if $H$ satisfies $int(H)\neq \emptyset, 0\notin cls(H)$, then  $\mu(\tau_H>n)\approx n^{1-\theta}$. 
\end{remark}

\begin{remark}
    A remark on the maximal large deviations which was not mentioned in \cite{mldp}: the paper \cite{mldp} proves that polynomial mixing rates imply polynomial maximal large deviations. However, if checking the proof of Lemma 3.3 and Proposition 3.4 of \cite{mldp} and the proof of Theorem 3.1 of \cite{mldp}, one can drop the condition of mixing rates and have the following new claim: $$\Big|\Big|\frac{\sum_{i\le N}\phi \circ f^i}{N}\Big|\Big|^{2p}_{2p} \precsim N^{-\beta} \text{ implies }\Big|\Big|\sup_{n \ge N}\Big| \frac{\sum_{i \le n}\phi \circ f^i}{n}\Big|\Big|\Big|^{2p}_{2p}\precsim N^{-\beta}.$$ where $\phi$ is an observable of a general dynamical system $(X, f, \mu)$, $\beta>0$ and $p$ is a positive integer. This can also be proved by the technique in \cite{sutams}. Note that mixing rates are not needed in this claim, we apply it to diophantine circle rotation and polynomial maximal large deviations for this system still hold.
\end{remark}

\section{Solenoids with intermittency}\label{highd}

In this section, we present a similar method to approach polynomial escape rates for Solenoids with intermittency. Let $\mathcal{M}=[0,1] \times \mathbb{D}$, $f(x,z)=(g_{\alpha}(x), \theta  z+e^{2\pi i x}/2)$, where 
\begin{enumerate}
    \item $g_{\alpha}$ is the Liverani-Saussol-Vaienti map (\ref{lsvmap})
    \item $\theta>0$ is so small that $\theta  ||Dg_{\alpha}||_{\infty}< 1-\theta$.
\end{enumerate}

It follows from \cite{Alves} that the attractor is $\bigcap_{i\ge 0}f^i(\mathcal{M})$, an SRB probability measure $\mu_{\alpha}$ exists on the attractor. 

\begin{theorem}[Escape Rates II]\ \label{thm1}\ \par
For the dynamical system $(X, f, \mu):=(\bigcap_{i\ge 0}f^i(\mathcal{M}), f, \mu_{\alpha})$, if a (dis)connected hole $H$ satisfies $ (\{0\}\times \mathbb{D})\bigcap cls(H)=\emptyset$ and $int(H)\neq \emptyset$, then  $\mu(\tau_H>n)\approx n^{1-1/\alpha}$.
 \end{theorem}

 \begin{proof} Choose a sufficiently small $a_m:=g_{\alpha}|_L^{-m}\{1\}$ (here $g_{\alpha}|_L$ is the left branch of $g_{\alpha}$) such that $H \subseteq (a_m,1]\times \mathbb{D}$. Define a hyperbolic product set $\Lambda:=([a_m,1] \times \mathbb{D}) \bigcap X$. Each stable manifold $\gamma^s$ is in the form of $\{x\}\times \mathbb{D}, x\in [a_m,1]$. The scheme of the proof is similar to that of Theorem \ref{thm}.  
 
     \vspace*{4pt}\noindent\textbf{1. A first return tower.} It follows from \cite{Sucmp} that the system $(X, f, \mu)$ can be modelled by a first return hyperbolic Young tower $\Delta:=\{(x,n) \in \Lambda \times \{0,1,2, \cdots\}: n<R(x)\}$. Here $R(x):=\inf\{n\ge 1: f^n(x)\in \Lambda\}$ defined on $\Lambda$ is a first return time with $\mu(R>n)\approx n^{1-1/\alpha}$, the dynamics  $F:\Delta \to \Delta$ sends $(x,n)$ to $(x,n+1)$ if $n+1<R(x)$, and $(x,n)$ to $(f^R(x),0)$ if $n=R(x)-1$. 
     
     It follows from \cite{Sucmp, Alves, Subbb} that  $\pi:\Delta \to X$ defined by $\pi(x,n)=f^n(x)$ is a bijection, a relation $\sim$ on $\Delta$ defined by \[(x,i) \sim (y,j) \text{ if and only if } i=j, x,y \text{ belong to a stable manifold }\gamma^s\]is an equivalence relation, and there is a quotient Young tower  $\widetilde{\Delta}:=\Delta/\sim$, a projection $\widetilde{\pi}_{\Delta}: \Delta\to \widetilde{\Delta}$, a quotient dynamics $\widetilde{F}: \widetilde{\Delta} \to \widetilde{\Delta}$, a probability $\mu_{\Delta}$ on $\Delta$ and a probability $\mu_{\widetilde{\Delta}}$ on $\widetilde{\Delta}$ satisfying  \begin{align}\label{youngrelations}
         \pi_* \mu_{\Delta}=\mu, \quad  \left(\widetilde{\pi}_{\Delta}\right)_{*}\mu_{\Delta}=\mu_{\widetilde{\Delta}},  \quad \widetilde{F}_{*}\mu_{\widetilde{\Delta}}=\mu_{\widetilde{\Delta}}, \quad F_{*}\mu_{\Delta}=\mu_{\Delta}, \quad \widetilde{\pi}_{\Delta} \circ F=\widetilde{F} \circ \widetilde{\pi}_{\Delta}.
     \end{align}\\
     \vspace*{4pt}\noindent\textbf{2. Lifting.} By (\ref{youngrelations}), $\mu(\tau_H>n)=\mu_{\Delta}(\inf\{n\ge 1: f^n \circ \pi \in H\}>n)=\mu_{\Delta}(\inf\{n\ge 1: F^n \in \pi^{-1}H\}>n)=\mu_{\Delta}(\inf\{n\ge 1: F^n \in H\times \{0\}\}>n)$.\\
     \vspace*{4pt}\noindent\textbf{3. Lower bounds.}  $\mu(\tau_H>n)\succsim n^{1-1/\alpha}$follows from the same arguments in step 3 of the proof in Theorem \ref{thm}.

     \vspace*{4pt}\noindent\textbf{4. Upper bounds.} Since $int(H)\neq \emptyset$, then there is a small disk $D\subseteq \mathbb{D}$ and a small Markov hole $I_1\subsetneq [a_m,1]$ (see its definition in step 4 of the proof of Theorem \ref{thm}) such that $I_1\times D\subsetneq H$. So there is a sufficiently large $N$ and a sufficiently small Markov hole $I\subsetneq [a_m,1]$ such that $g_{\alpha}(I)=I_1$ and $f^N(I\times \mathbb{D})\subseteq I_1\times D\subsetneq H$. Then by (\ref{youngrelations}) we have $\mu(\tau_H>n) $\begin{align*}
          &\le \mu(\inf\{n\ge 1:f^n\in  f^N(I\times \mathbb{D})\}>n)\\
         &=\mu(\inf\{n\ge 1:f^n\in  I\times \mathbb{D}\}>n)\\
         & =\mu_{\Delta}(\inf\{n\ge 1:F^n\in (I\times \mathbb{D})\times \{0\}\}>n)\\
         &  =\mu_{\Delta}(\inf\{n\ge 1:F^n\in \widetilde{\pi}_{\Delta}^{-1}(I\times \{0\})\}>n)=\mu_{\widetilde{\Delta}}(\inf\{n\ge 1:\widetilde{F}^n\in I\times \{0\}\}>n).
     \end{align*} Now we repeat the arguments of steps 4 and 5 in the proof of Theorem \ref{thm}, by replacing $H$, $F$ and $\Delta$ there with $I$, $\widetilde{F}$ and $\widetilde{\Delta}$, respectively. Thus we have \[\mu(\tau_H>n)\le \mu_{\widetilde{\Delta}}(\inf\{n\ge 1:\widetilde{F}^n\in I\times \{0\}\}>n)\precsim n^{1-1/\alpha}.\]
Hence, we conclude the proof.
     \end{proof}
     \begin{remark}\label{adapt2}
         The method in this section can be easily extended to various hyperbolic systems, e.g. replacing $g_{\alpha}$ by Farey maps (\ref{farey}) gives a new three-dimensional hyperbolic system with intermittency, a similar argument gives $\mu(\tau_H>n)\approx n^{1-\theta}$ when $(\{0\}\times \mathbb{D}) \bigcap cls(H)=\emptyset$ and $int(H)\neq \emptyset$.
     \end{remark}
     \begin{remark}
       Bunimovich focusing billiards (e.g., Bunimovich stadiums, Bunimovich flowers) are typical examples of hyperbolic systems with intermittency. However, we will not present a decay rate of $\mu(\tau_H>n)$ in this paper. This is because studying a stronger property ``an explicit asymptotic expansion for the  $\mu(\tau_H>n)$ for Bunimovich billiards" reveals many interesting facts in chaos (see \cite{chaos}). We will address this new question for this type of billiards in a separate paper.
     \end{remark}

\section*{Declarations}
\subsection*{\textbf{Conflict of interest} This manuscript has no conflict of interest}

\subsection*{\textbf{Open Access} This article is licensed under a Creative Commons Attribution 4.0 International License, which
permits use, sharing, adaptation, distribution and reproduction in any medium or format, as long as you give
appropriate credit to the original author(s) and the source, provide a link to the Creative Commons licence,
and indicate if changes were made. The images or other third party material in this article are included in the
article’s Creative Commons licence, unless indicated otherwise in a credit line to the material. If material is
not included in the article’s Creative Commons licence and your intended use is not permitted by statutory
regulation or exceeds the permitted use, you will need to obtain permission directly from the copyright holder.
To view a copy of this licence, visit http://creativecommons.org/licenses/by/4.0/.}

\subsection*{\textbf{Data Availability Statement} The authors declare that the data is not presented in this study.} 

\bibliography{bibtext}

\end{document}